\documentclass[journal,twoside,web]{ieeecolor}
\usepackage{amsmath,amssymb,amsfonts}

\def\logoname{LOGO-generic-web}
\definecolor{subsectioncolor}{rgb}{0,0.541,0.855}
\def\journalname{IEEE Transactions on Automatic Control}

\def\BibTeX{{\rm B\kern-.05em{\sc i\kern-.025em b}\kern-.08em
    T\kern-.1667em\lower.7ex\hbox{E}\kern-.125emX}}

\usepackage[T1]{fontenc}
\usepackage{comment}

\usepackage{tikz}
\usepackage{pgfplots}
\usepackage{subfig}
\usepackage{caption}
\usepackage{siunitx}

\usepackage
{
    graphicx,
    dsfont, 
    xcolor,
    mathtools,
    authblk,
    nicefrac,
    array,
    enumerate,
    subfig
}

\usepackage[utf8]{inputenc}
\usepackage{dblfloatfix}

\usepackage[pdffitwindow=false,
            plainpages=false,
            pdfpagelabels=true,
            pdfpagemode=UseOutlines,
            pdfpagelayout=SinglePage,
            bookmarks=false,
            colorlinks=true,
            hyperfootnotes=false,
            linkcolor=blue,
            urlcolor=blue!30!black,
            citecolor=green!50!black]{hyperref}

\graphicspath{{Figures/}}

\newcommand{\R}{\mathbb{R}}                                     

\newcommand{\N}{{\mathbb{N}}}                                   
\newcommand{\ddt}{\tfrac{\text{\normalfont d}}{\text{\normalfont d}t}} 

\newcolumntype{C}[1]{>{\centering\let\newline\\\arraybackslash\hspace{0pt}}m{#1}}

\DeclareMathOperator{\sgn}{sgn}
\DeclareMathOperator{\erf}{erf}

\newcommand{\ds}[1]{{\rm \, d} #1 \,}

\newtheorem{theorem}{Theorem}[section]

\newtheorem{assumption}[theorem]{Assumption}
\newtheorem{remark}[theorem]{Remark}

\newcommand{\cD}{\mathcal{D}}

\newcommand{\cF}{\mathcal{F}}

\newcommand{\eps}{\varepsilon}

\newcommand{\setdef}[2]{\left\{\, #1 \left|\, \vphantom{#1} #2\right.\right\}}

\DeclareOldFontCommand{\rm}{\normalfont\rmfamily}{\mathrm}

\makeatletter
\renewcommand*\env@matrix[1][*\c@MaxMatrixCols c]{%
  \hskip -\arraycolsep
  \let\@ifnextchar\new@ifnextchar
  \array{#1}}
\makeatother

\mathtoolsset{centercolon} 

\makeatletter
\def\blfootnote{\gdef\@thefnmark{}\@footnotetext}
\makeatother

\begin{document}

\title{String stability and guaranteed safety via funnel cruise control for vehicle platoons}
\author{Thomas Berger\thanks{Funded by the Deutsche Forschungsgemeinschaft (DFG, German
Research Foundation) -- Project-ID 524064985.} and Bart Besselink
\thanks{Thomas Berger is with the Universit\"at Paderborn, Institut f\"ur Mathematik, Warburger Str.~100, 33098~Paderborn, Germany (e-mail: thomas.berger@math.upb.de).}
\thanks{Bart Besselink is with the Bernoulli Institute for Mathematics, Computer Science and Artificial Intelligence, University of Groningen, Nijenborgh 9, 9747 AG Groningen, Netherlands (e-mail: b.besselink@rug.nl).}}

\maketitle

\begin{abstract}
We study decentralized control strategies for  platoons of autonomous vehicles with heterogeneous and nonlinear dynamics. Based on ideas from funnel control, we present a novel decentralized control algorithm which is able to guarantee a safety distance between any two vehicles, a good traffic flow and it achieves string stability of the controlled platoon. We illustrate the performance of the controller by simulations of two extreme scenarios.
\end{abstract}

\begin{IEEEkeywords}
vehicle platoons, autonomous driving, guaranteed safety, string stability, decentralized control, funnel control.
\end{IEEEkeywords}

\section{Introduction}\label{sec:introduction}

\IEEEPARstart{T}{o} optimize the traffic flow on highways~\cite{HedrTomi94} and to reduce the fuel consumption by minimizing air drag~\cite{BonnFrit00}, the aggregation of autonomous vehicles in platoons seems favorable. To guarantee the safety of individual vehicles despite the small distances in a platoon, an automated control is necessary, for which this property can be proved in a rigorous way.

Research on vehicle platooning has a long history, see \cite{LeviAtha66,Pepp74} for early works, and has received considerable attention in the literature, see \cite{varaiya_1993,IoanChie93,SwarHedr99,AlamBess15,ploeg_2014b,zheng_2016} for a selection of key contributions and \cite{guanetti_2018,ersal_2020} for recent surveys.
Due to the limited capacities of communication between the vehicles in a platoon, most classical studies focussed on decentralized control techniques, see e.g.~\cite{Pepp74,Lunz92}, where each vehicle only has the local information of the distance to the preceding vehicle and its own velocity. Sometimes, additional information such as the velocity of the leader vehicle is required, see~\cite{PeteMidd14,SeilPant04}, but it is difficult to communicate this global information to each member of the platoon. Furthermore, decentralized controllers are cost efficient since no expensive hardware for communication must be installed; they are reliable, since a failure of the hardware in a single vehicle does not threaten the whole platoon; they are flexible, since vehicles in the platoon can easily change their positions and additional vehicles can be integrated; etc. Different from decentralized control, \textit{distributed} control is especially popular in the model predictive control literature~\cite{dunbar_2006,CampJia02}, however a certain exchange of information is required there. In this work, we focus on the development of a \textit{decentralized} control strategy, i.e.\ we assume that there is no communication between neighboring vehicles (and controllers).

A crucial property for controlled platoons is string stability, which essentially expresses that disturbances are not amplified when propagated through the platoon. In particular, a deceleration of the leader vehicle should not lead to a so called ghost traffic jam: Each follower decelerates a bit more than it's preceding vehicle so that eventually the platoon comes to a standstill. In~\cite{SwarHedr96} string stability is introduced as Lyapunov stability of the origin, when we interpret the platoon as an interconnected system. In the case of linear vehicle dynamics, string stability can be characterized in terms of the transfer functions of each member of the platoon~\cite{PeteMidd14,SeilPant04}. However, it was shown in~\cite{SeilPant04} that with linear controllers, using only local information and a constant spacing policy, it is impossible to achieve string stability, see \cite{wijnbergen_2020} for related results. Therefore, some approaches focussed on additionally allowing for communication between the vehicles, leading to the concept of cooperative adaptive cruise control (CACC)~\cite{NausVugt10,OencPloe14}. A drawback is that the advantages of a decentralized control (cost efficiency, reliability, flexibility, etc.) are lost then. Therefore, in the present paper we focus on the alternative of a nonlinear controller and in doing so we also allow for nonlinear vehicle dynamics.

In the case of nonlinear models, alternatives to the frequency domain approach are requisite. An appropriate concept, which also incorporates the influence of external disturbances, is disturbance string stability introduced in~\cite{BessJoha17,besselink_2018}. Essentially, this is a uniform (with respect to the vehicle index) input-to-state stability, cf.~\cite{Sont89a}. In the literature, a couple of other modifications of string stability are available, see~\cite{FengZhan19} for an overview. In the present paper, we introduce a practical version of disturbance string stability for the velocities of the vehicles, which measures the effect of variations of the leader velocity on the velocities of the other vehicles in the platoon.

In order to satisfy the requirements on the safety of the vehicles in the platoon, we develop a novel control design which uses ideas from funnel control. The concept of funnel control was developed in the seminal work~\cite{IlchRyan02b} (see also the recent survey in~\cite{BergIlch21}) and proved advantageous in a variety of applications such as control of industrial servo-systems~\cite{Hack17}, underactuated multibody systems~\cite{BergDrue21,DrueLanz24}, electrical circuits~\cite{BergReis14a,SenfPaug14}, peak inspiratory pressure~\cite{PompWeye15} and a moving water tank~\cite{BergPuch22}. Funnel control for the case of two vehicles following each other has been considered in~\cite{BergRaue18,BergRaue20} and the \textit{funnel cruise controller} has been introduced to guarantee a safe following; a related work can be found in~\cite{verginis_2018}. A drawback of these approaches is that, when implemented in each vehicle of a platoon, they do not achieve string stability. A control strategy guaranteeing both a prescribed performance and string stability is developed in~\cite{GuoLi19}, however a lumped tracking error (a linear combination of position and velocity differences) is considered, from the prescribed performance of which safety of each vehicle cannot be inferred. Furthermore, string stability is only shown for this lumped tracking error and each vehicle requires knowledge of the position and velocity of the leader vehicle, thus necessitating reliable inter-vehicle communication.

In the present work, we present a new control design, which is different from the funnel cruise controller (and related approaches) and achieves the objectives: guaranteed safety distance between any two vehicles, good traffic flow, string stability of the controlled platoon, and decentralized implementation based on local information.

\subsection{Nomenclature}

In the following let $\N$ denote the natural numbers, $\N_0 = \N \cup\{0\}$, and $\R_{\ge 0} =[0,\infty)$. By $\|x\|$ we denote the Euclidean norm of $x\in\R^n$. For some interval $I\subseteq\R$, some $V\subseteq\R^m$ and $k\in\N$, $L^\infty(I, \R^{n})$ is the Lebesgue space of measurable, essentially bounded functions $f\colon I\to\R^n$, $W^{k,\infty}(I, \R^{n})$ is the Sobolev space of all functions $f:I\to\R^n$ with $k$-th order weak derivative $f^{(k)}$ and $f,f^{(1)},\ldots,f^{(k)}\in L^\infty(I, \R^{n})$, and $C^k(V,  \R^{n})$ is the set of $k$-times continuously differentiable functions $f: V \to \R^{n}$, with $C(V,  \R^{n}) := C^0(V, \R^{n})$.

\subsection{Vehicle dynamics}

\begin{figure*}
	\begin{center}
\resizebox{\textwidth}{!}{
		\begin{tikzpicture}
            \node (vehN) at (-100mm,0mm) [rectangle,draw=none,inner sep=0mm] {\scalebox{-1}[1]{\includegraphics[width=34mm]{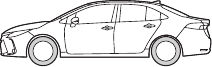}}};
			\node (veh1) at (-40mm,0mm) [rectangle,draw=none,inner sep=0mm]
            {\scalebox{-1}[1]{\includegraphics[width=34mm]{toyota_toleft.pdf}}};
			\node (veh0) at ( 24mm,0mm) [rectangle,draw=none,inner sep=0mm]
            {\scalebox{-1}[1]{\includegraphics[width=34mm]{toyota_toleft.pdf}}};

            \draw (vehN.south west) -- (veh0.south east);
            \node (dots) at (vehN.east) [inner sep=0mm, xshift = 15mm] {$\cdots$};
            \node (help1) at (veh0.west) [inner sep=0mm, xshift = -13mm] {};
            \draw[dotted,line width=0.8pt] (help1) -- +(0mm,-5mm);
            \draw [<->,line width=0.6pt] (help1)--(veh0.west) node[midway, above]{$d_{\min}$};

            \node (help2) at (veh0.west) [inner sep=0mm, yshift = 10mm] {};
            \node (help3) at (veh1.west) [inner sep=0mm, xshift = -5mm, yshift = 10mm] {};
            \draw [<->,line width=0.6pt] (help2)--(help3) node[midway, above]{$d_{\max}$};
            \draw[dotted,line width=0.8pt] (help2) -- +(0mm,-16mm);
            \draw[dotted,line width=0.8pt] (help3) -- +(0mm,-16mm);

            \draw[->,line width=0.6pt] (vehN.east) -- +(8mm,0mm)  node[midway, above]{$v_N(t)$};
            \draw[->,line width=0.6pt] (veh1.east) -- +(8mm,0mm)  node[midway, above]{$v_1(t)$};
            \draw[->,line width=0.6pt] (veh0.east) -- +(8mm,0mm)  node[midway, above]{$v_0(t)$};

            \node (help4) at (vehN.south west) [inner sep=0mm, xshift = -12mm, yshift = -5mm] {};
            \node (help5)  [below of = help4, inner sep=0mm,  yshift = 5mm] {};
            \node (help6)  [below of = help5, inner sep=0mm, yshift = 5mm] {};
            \draw[dotted,line width=0.8pt] (vehN.south west) -- +(0mm,-5mm);
            \draw[dotted,line width=0.8pt] (veh1.south west) -- +(0mm,-10mm);
            \draw[dotted,line width=0.8pt] (veh0.south west) -- +(0mm,-15mm);

            \draw[->,line width=0.6pt] (help4) -- +(12mm,0mm)  node[midway, above]{$x_N(t)$};
            \draw[->,line width=0.6pt] (help5) -- +(72mm,0mm)  node[midway, above]{$x_1(t)$};
            \draw[->,line width=0.6pt] (help6) -- +(136mm,0mm)  node[pos=0.7, above]{$x_0(t)$};
            \draw[line width=0.6pt] (help4.south east) -- +(0mm,1mm); \draw[line width=0.6pt] (help4.north east) -- +(0mm,-1mm);
            \draw[line width=0.6pt] (help5.south east) -- +(0mm,1mm); \draw[line width=0.6pt] (help5.north east) -- +(0mm,-1mm);
            \draw[line width=0.6pt] (help6.south east) -- +(0mm,1mm); \draw[line width=0.6pt] (help6.north east) -- +(0mm,-1mm);

		\end{tikzpicture}
}
	\end{center}
	\caption{Framework for the control of a vehicle platoon}
\label{Fig:platoon}
\end{figure*}

Consider a platoon of $N$ heterogeneous vehicles\footnote{In this work, the leader vehicle is not counted as a member of the platoon. It assumes a special role, without any dynamics describing its behavior.} with nonlinear dynamics
\begin{equation}\label{eq:Sys}
\begin{aligned}
    \dot x_i(t) &= v_i(t),\\
    m_i \dot{v}_i(t) &= u_i(t) - f_i(t,x_i(t),v_i(t)) + d_i(t),\quad i=1,\ldots,N,
\end{aligned}
\end{equation}
where $x_i(t)\in\R$ and $v_i(t)\in\R$ denote the position and velocity, respectively, of vehicle $i$ at time $t\geq0$, see also Fig.~\ref{Fig:platoon}. Moreover, $m_i$ (in \si{\kilo\gram}) is the mass of vehicle~$i$, $d_i\in L^\infty(\R_{\ge 0},\R)$ is a bounded disturbance (capturing modelling errors, uncertainties and noises), and the nonlinear function $f_i$ is the sum of the forces due to gravity~$F_{i,g}$, the aerodynamic drag~$F_{i,a}$ and the rolling friction~$F_{i,r}$, that is
\[
    f_i(t,x,v) = F_{i,g}(x) + F_{i,a}(t,x,v) + F_{i,r}(v),
\]
with
\begin{align*}
  F_{i,g}:&\ \R\to\R,\ x\mapsto m_i\, g \sin \theta_i(x),\\
  F_{i,a}:&\ \R_{\ge 0}\times \R\times \R\to\R,\\
  & (t,x,v)\mapsto \tfrac12 \rho_i(t,x) C_{i,d}\, A_i \sgn(v) v^2,\\
  F_{i,r}:&\ \R\to\R,\ v\mapsto m_i\, g\, C_{i,r} \sgn(v).
\end{align*}
Here, $g = \SI{9.81}{\metre\per\second^2}$ is the acceleration of gravity,~$\theta_i(x)\in [-\tfrac{\pi}{2}, \tfrac{\pi}{2}]$ (in \si{\radian}) and $\rho_i(t,x)$ (in \si{\kilo\gram\per\metre^3}) denote the slope of the road and the (bounded) density of air at time~$t$ and location~$x$ for vehicle~$i$, resp.,~$C_{i,d}$ denotes the (dimensionless) shape-dependent aerodynamic drag coefficient,~$C_{i,r}$ the (dimensionless) coefficient of rolling friction, and~$A_i$ (in \si{\metre^2}) the frontal area of vehicle~$i$, respectively. The control input~$u_i$ of each vehicle is the force resulting from the contact of the wheels with the road and generated by the engine of the vehicle.

A discontinuous rolling friction causes problems in the theoretical treatment. Therefore, we approximate the~$\sgn$ function in~$F_{i,r}$ by the smooth error function
\[
    \erf(z) = \frac{2}{\sqrt{\pi}} \int_0^z e^{-t^2}\ds{t},\quad z\in\R,
\]
using that $\lim_{\alpha\to \infty} \erf(\alpha z) = \sgn(z)$ for all $z\in\R$. We will thus use the following model for the rolling friction:
\begin{equation}\label{eq:friction}
     F_{i,r}:\R\to\R,\ v\mapsto m_i\, g\, C_{i,r} \erf(\alpha v)
\end{equation}
for sufficiently large parameter~$\alpha>0$. For other friction models see~\cite{ArmsDupo94,LeinNijm04}.

The initial conditions for~\eqref{eq:Sys} are
\begin{equation}\label{eq:IC}
  x_i(0) = x_i^0\in\R,\ \ v_i(0) = v_i^0\in\R,\ \ x_i^0 < x_{i-1}^0,\ i=1,\ldots,N,
\end{equation}
where $x_0^0 = x_0(0)\in\R$ is the initial position of the leader vehicle which has a position trajectory $x_0 \in C^2(\R_{\ge 0},\R)$ and we set $v_0 := \dot x_0$.

\subsection{Control objective}

The control objective is to
\begin{enumerate}[(O1)]
  \item guarantee a safety distance between any two vehicles,
  \item ensure a good traffic flow (distances between vehicles don't get too large),
  \item achieve string stability of the controlled platoon,
  \item only use decentralized controllers based on local information.
\end{enumerate}

These objectives can be formalized in the following way. Let a desired safety distance $d_{\min}>0$ and a desired maximal distance $d_{\max}>0$ of the vehicles be given, then it should hold that
\[
    \forall\, i=1,\ldots,N\ \forall\, t\ge 0:\ d_{\min} < x_{i-1}(t) - x_i(t) < d_{\max}.
\]
This ensures objectives~(O1) and~(O2) and is also illustrated in Fig.~\ref{Fig:platoon}. For string stability as in~(O3) we view the velocity profile~$v_0$ of the leader vehicle as the disturbance and require that the velocity of any other vehicle is linearly bounded by~$v_0$, uniformly in~$i=1,\ldots,N$ and independent of the platoon length~$N$. Additionally, we allow for an offset, which is also uniform in~$i$ and~$N$, thus calling this property \textit{practical velocity string stability}:
\begin{multline}\label{eq:string-stab}
\exists\, C_1,C_2>0\ \forall\, N\in\N\ \forall\, i=1,\ldots,N\ \forall\, v_0\in L^\infty(\R_{\ge 0},\R):\quad \\
 \|v_i\|_\infty \le C_1 + C_2 \|v_0\|_\infty.
\end{multline}
Here~$v_i$ denotes the solution of~\eqref{eq:Sys} under an appropriate feedback~$u_i$ and for relevant classes of initial conditions and external disturbances, all of  which will be defined below.

One may wonder why only a linear estimate for~$v_i$ in terms of~$v_0$ is required here, while the dynamics~\eqref{eq:Sys} are nonlinear. The reason is that the nonlinear part due to~$F_{i,a}$ always has a decelerating effect and hence influences the estimates only in a ``good'' way, so that a linear estimate is possible in the end.

To achieve~(O4), in Section~\ref{Sec:ContrStruc} we will present a decentralized control design, which for vehicle~$i$ only requires the instantaneous measurements of the distances $x_{i-1}(t) - x_i(t)$, the velocity~$v_i(t)$ and the relative velocity $v_{i-1}(t) - v_i(t)$. Although the latter quantity is not as directly available as the other two, it is local information and can be measured by vehicle~$i$ e.g.\ using a radar speed gun or modern LIDAR devices for instance. Apart from these measurements, the controller will require no exact knowledge of any of the system parameters or initial values.

\subsection{Organization of the present paper}

The paper is structured as follows. In Section~\ref{Sec:ContrStruc} we present the decentralized controller, which aims at achieving the control objectives (O1)--(O4) introduced above. The control design exploits ideas from funnel control. Feasibility of the control is proved in the main result in Section~\ref{Sec:Main}: it is shown that safety and a good traffic flow are guaranteed (objectives~(O1) and~(O2)) and practical velocity string stability is achieved (objective~(O3)) by the decentralized controller  (objective~(O4)). The results are illustrated by two different simulation scenarios for  inhomogeneous platoons in Section~\ref{Sec:Sim}. The paper is concluded by Section~\ref{Sec:Concl}.

\section{Decentralized control design}\label{Sec:ContrStruc}

In this section, we will introduce a controller aimed at achieving the control objectives (O1)--(O4). As a first step, let
\begin{align}
\xi_i(t) &= x_i(t) - x_{i-1}(t) + d_{\min},
\end{align}
denote the difference between the inter-vehicle distance and the safety distance, and introduce
\begin{align}
w_i(t) &= v_i(t) - v_{i-1}(t) - \frac{1}{\xi_i(t)} - \frac{1}{M + \xi_i(t)}.
\end{align}
Recall that (O1) and (O2) require vehicle $i$ to satisfy $d_{\min} < x_{i-1}(t) - x_i(t) < d_{\max}$, which is equivalent to $-M < \xi_i(t) < 0$, with
\begin{align}
	M := d_{\max} - d_{\min}.
\end{align}
Note that the definition of $w_i(t)$ is such that $w_i(t)\to\infty$ whenever $\xi_i(t)\nearrow 0$ and $w_i(t)\to-\infty$ whenever $\xi_i(t)\searrow -M$. As $w_i(t)$ characterizes safety, we would like to ensure its boundedness by forcing it to evolve in a performance funnel
\[
    \cF_\psi = \setdef{(t,w)\in\R_{\ge 0}\times\R}{ |w| < \psi(t)},
\]
defined by a function $\psi$ belonging to the set of admissible funnel boundaries
\[
    \Psi = \setdef{\psi\in W^{1,\infty}(\R_{\ge 0},\R)}{\!\!\begin{array}{l} \text{$\psi(t)>0$ for all $t\ge 0$},\\ \liminf_{t\to\infty} \psi(t)>0\end{array}\!\!},
\]
see also Fig.~\ref{Fig:funnel}. Evolution in the performance funnel, that is $(t,w_i(t))\in \cF_\psi$ for all $t\ge 0$, thus guarantees objective~(O1) and~(O2). The practical velocity string stability property~\eqref{eq:string-stab} will be inferred from the combination of the funnel control law with a constant headway-like spacing policy (i.e., including a dependence on the velocities~$v_i$ and~$v_{i-1}$) as
\[
    u_i(t) = -k_1 (v_i(t) - v_{i-1}(t)) - k_2 e_i(t) - k_{i,3}(t) w_i(t),
\]
where $k_1, k_2$ are positive gain parameters,~$k_{i,3}(t) = 1/\big(\psi(t) - |w_i(t)|\big)$ is the funnel gain and $e_i(t) = \xi_i(t) + \lambda v_i(t)$ with $\lambda >0$ is an error variable related to the constant headway policy. The latter is known for its inherent attenuation of disturbances~\cite{BessJoha17} and will hence ensure attainment of objective~(O3). More precisely, if $e_i(\cdot)$ is constant, then $\dot e_i(t) = 0$ for $t\ge 0$, which gives
\[
    \dot v_i(t) = -\frac{1}{\lambda} v_i(t) + \frac{1}{\lambda} v_{i-1}(t)
\]
and describes how disturbances of the leader velocity are attenuated through the platoon. The overall decentralized design of the controller (and, thus, attainment of objective~(O4)) is summarized in the following controller for vehicle platoons~\eqref{eq:Sys}:

\begin{figure*}[!b]
\begin{equation}\label{eq:def-K}
\begin{aligned}
  K: & \setdef{(t,x,v,\hat v)\in\R_{\ge 0}\times \R^3}{\begin{array}{l} -d_{\min} > x > -d_{\max},\ \\
     \left|v - \hat v - \frac{1}{x + d_{\min}} - \frac{1}{x + d_{\max}}\right| < \psi(t)
    \end{array}}\to\R,\\
  & (t,x,v,\hat v)\mapsto -k_1 (v - \hat v) - k_2(x + d_{\min} + \lambda v) - \frac{v - \hat v - \frac{1}{x + d_{\min}} - \frac{1}{x + d_{\max}}}{\psi(t) - \left|v - \hat v - \frac{1}{x + d_{\min}} - \frac{1}{x + d_{\max}}\right|}.
\end{aligned}
\end{equation}
\end{figure*}

\begin{equation}\label{eq:SSFCC}
\boxed{
\begin{aligned}
    \xi_i(t) &= x_i(t) - x_{i-1}(t) + d_{\min},\\
    e_i(t) &= \xi_i(t) + \lambda v_i(t),\\
    w_i(t) &= v_i(t) - v_{i-1}(t) - \frac{1}{\xi_i(t)} - \frac{1}{M + \xi_i(t)},\\
    k_{i,3}(t) &=  \frac{1}{\psi(t) - |w_i(t)|},\\
    u_i(t) &= -k_1 (v_i(t) - v_{i-1}(t)) - k_2 e_i(t) - k_{i,3}(t) w_i(t),
\end{aligned}
}
\end{equation}
with the controller design parameters
\begin{equation}\label{eq:FC-param}
\boxed{
\begin{aligned}
& d_{\max} > d_{\min}> 0,\quad M:=d_{\max} - d_{\min},\\
& \lambda, k_1, k_2 > 0,\quad \psi\in\Psi.
\end{aligned}
}
\end{equation}

Note that, for given funnel boundary $\psi$, the controller~\eqref{eq:SSFCC} only depends on the inter-vehicle distance $x_i - x_{i-1}$ and the velocities $v_i$ and $v_{i-1}$ of the controlled vehicle and its predecessor, respectively. For later use we introduce the brief presentation of the feedback as
\begin{equation}\label{eq:SSFCC-K}
  u_i(t) = K\bigl(t,x_i(t)-x_{i-1}(t),v_i(t),v_{i-1}(t)\bigr),
\end{equation}
where the function~$K$ is formally defined as in~\eqref{eq:def-K}.

\begin{figure}[h]
\begin{center}
\begin{tikzpicture}[scale=0.45]
\tikzset{>=latex}
  \filldraw[color=gray!25] plot[smooth] coordinates {(0.15,4.7)(0.7,2.9)(4,0.4)(6,1.5)(9.5,0.4)(10,0.333)(10.01,0.331)(10.041,0.3) (10.041,-0.3)(10.01,-0.331)(10,-0.333)(9.5,-0.4)(6,-1.5)(4,-0.4)(0.7,-2.9)(0.15,-4.7)};
  \draw[thick] plot[smooth] coordinates {(0.15,4.7)(0.7,2.9)(4,0.4)(6,1.5)(9.5,0.4)(10,0.333)(10.01,0.331)(10.041,0.3)};
  \draw[thick] plot[smooth] coordinates {(10.041,-0.3)(10.01,-0.331)(10,-0.333)(9.5,-0.4)(6,-1.5)(4,-0.4)(0.7,-2.9)(0.15,-4.7)};
  \draw[thick,fill=lightgray] (0,0) ellipse (0.4 and 5);
  \draw[thick] (0,0) ellipse (0.1 and 0.333);
  \draw[thick,fill=gray!25] (10.041,0) ellipse (0.1 and 0.333);
  \draw[thick] plot[smooth] coordinates {(0,2)(2,1.1)(4,-0.1)(6,-0.7)(9,0.25)(10,0.15)};
  \draw[thick,->] (-2,0)--(12,0) node[right,above]{\normalsize$t$};
  \draw[thick,dashed](0,0.333)--(10,0.333);
  \draw[thick,dashed](0,-0.333)--(10,-0.333);
  \node [black] at (0,2) {\textbullet};
  \draw[->,thick](4,-3)node[right]{\normalsize$\mu$}--(2.5,-0.4);
  \draw[->,thick](3,3)node[right]{\normalsize$(0,w(0))$}--(0.07,2.07);
  \draw[->,thick](9,3)node[right]{\normalsize$\psi(t)$}--(7,1.4);
\end{tikzpicture}
\end{center}
\caption{Error evolution in a funnel $\mathcal F_{\psi}$ with boundary $\psi(t)$.}
\label{Fig:funnel}
\end{figure}

By the properties of~$\Psi$ there exists $\mu>0$ such that $\psi(t)\ge \mu$ for all $t\ge 0$. It is important to note that the function $\psi\in\Psi$ is a design parameter in the control law~\eqref{eq:SSFCC} and its choice is up to the designer. Although~$\psi$ does not need to be monotonically decreasing in general, it is usually convenient to choose it of the form
\[
    \psi(t) = \alpha e^{-\beta t} + \gamma,\quad  \alpha\ge 0,\ \beta, \gamma >0.
\]
Other typical choices for funnel boundaries are outlined in~\cite[Sec.~3.2]{Ilch13}.

One may observe that the control law~\eqref{eq:SSFCC} introduces several potential singularities in the closed-loop differential equation, for instance at $\xi_i(t) = 0$, $\xi_i(t) = -M$ and $|w_i(t)| = \psi(t)$. It will be shown in the main result in the following section that any solution satisfies $-M < \xi_i(t) < 0$ or, what is the same, $d_{\min} < x_{i-1}(t) - x_i(t) < d_{\max}$ and additionally $|w_i(t)| < \psi(t)$. Therefore, the reciprocal terms in~\eqref{eq:SSFCC} guarantee the control objectives~(O1) and~(O2) and thus ensure a safe operation of the vehicle platoon.

\begin{figure*}[!b]
\begin{equation}\label{eq:cD}
    \cD := \setdef{(t,x,v)\in\R_{\ge 0}\times \R^N\times\R^N}{\begin{array}{l} d_{\min} < x_0(t) - x_1 < d_{\max},\ d_{\min} < x_{i-1} - x_i < d_{\max},\\
     \left|v_1 - v_0(t) - \frac{1}{x_1 - x_0(t) + d_{\min}} - \frac{1}{x_1 - x_0(t) + d_{\max}}\right| < \psi(t),\\[2mm]
     \left|v_i - v_{i-1} - \frac{1}{x_i - x_{i-1} + d_{\min}} - \frac{1}{x_i - x_{i-1} + d_{\max}}\right| < \psi(t),\\
     i=2,\ldots,N
    \end{array}},
\end{equation}
\end{figure*}

\section{Main result}\label{Sec:Main}

Before stating the main result we introduce some assumptions, necessary for the proof of feasibility of~\eqref{eq:SSFCC}.

\begin{assumption}\label{Ass1}
The functions $\theta_i$ and $\rho_i$ are continuous for all $i=1,\ldots,N$. Furthermore, there exist $\bar d>0$,  $\bar \rho > 0$ and $\bar m > 0$ such that for all $N\in\N$, for all $i=1,\ldots,N$, for all $t\ge 0$ and for all $x\in\R$ we have that
\begin{align*}
  |F_{i,g}(x) + F_{i,r}(x) + d_i(t)| &\le \bar d,\\
    m_i &\le \bar m,\\
   0\le \tfrac12 \rho_i(t,x)\, C_{i,d}\, A_i &\le \bar\rho.
\end{align*}
\end{assumption}

\begin{assumption}\label{Ass2}
The leader position trajectory $x_0 \in C^2(\R_{\ge 0},\R)$ is such that $v_0 := \dot x_0$ and $\dot v_0$ are bounded. Furthermore, there exists $\delta >0$ (depending on $x_0(0)$ and $v_0(0)$) such that for all $N\in\N$ and for all $i=1,\ldots,N$ we have that $-M + \delta \le \xi_i(0) \le -\delta$ and $|w_i(0)| \le \psi(0) - \delta$. The initial velocities are bounded by $|v_i(0)| \le M/\lambda$ for all $i=1,\ldots,N$ and all $N\in\N$.
\end{assumption}

While the above assumptions are standard and can always be satisfied in any platoon, the following assumption is of a more technical nature and required for the proof. It essentially states that, when the platoon length would become infinitely long, then the masses of the vehicles must be monotonically decreasing to zero from a certain point on.

\begin{assumption}\label{Ass3}
There exist $p,q\in (0,1)$ with $(1+p)q<1$ and $N_0\in\N$ such that for all $N\ge N_0$ and all $i=N_0,\ldots, N$ we have that $|m_i - m_{i-1}| \le p\, m_i$ and $m_i \le q\, m_{i-1}$.
\end{assumption}

Since in practice all platoons have finite length, the above assumption is always satisfied with $N_0$ being the number of existing vehicles on earth. As mentioned above, it is simply required for technical reasons. Furthermore, from a mathematical point of view, it provides some insight into the required structure of platoons with arbitrary length. Note that it is a consequence of Assumption~\ref{Ass3} that $q < p$.

Still, one might argue that Assumption~\ref{Ass3} is unreasonable from a practical point of view. It is expected that the assumption can be avoided when all vehicles in the platoon have access to some common information (such as lead vehicle position and velocity) as in, e.g., \cite{BessJoha17}, but a detailed study of this topic is left for future work.

The feasibility proof of the controller~\eqref{eq:SSFCC} for~\eqref{eq:Sys} requires the notion of a solution. For a platoon length~$N$, $(x,v) = (x_1,\ldots,x_N,v_1,\ldots,v_N):[0,\omega)\to\R^{2N}$, $\omega\in(0,\infty]$, is called a solution of~\eqref{eq:Sys},~\eqref{eq:SSFCC}, if the initial conditions~\eqref{eq:IC} hold and $(x,v)$ is locally absolutely continuous and satisfies the differential equation in~\eqref{eq:Sys} with $u_1,\ldots,u_N$ as in~\eqref{eq:SSFCC} for almost all $t\in[0,\omega)$; $(x,v)$ is called maximal, if it has
no right extension that is also a solution. Note that uniqueness of solutions of~\eqref{eq:SSFCC} for~\eqref{eq:Sys} is not guaranteed in general.

We are now in the position to state the main result of this paper.

\begin{theorem}\label{thm:main}
Consider a platoon of~$N$ vehicles with dynamics~\eqref{eq:Sys} and initial conditions~\eqref{eq:IC}, where $x_0 \in C^2(\R_{\ge 0},\R)$ is the position of the leader vehicle and $v_0 := \dot x_0$. Furthermore, let Assumptions~\ref{Ass1}--\ref{Ass3} hold. Then there exists a sufficiently large $k_2>0$ (independent of~$N$) such that the controller~\eqref{eq:SSFCC} with parameters~\eqref{eq:FC-param} applied to~\eqref{eq:Sys} yields a closed-loop system which has a solution, and every solution can be extended to a maximal solution $(x_1,\ldots,x_N,v_1,\ldots,v_N):[0,\omega)\to\R^{2N}$, $\omega\in(0,\infty]$, which has the properties:
\begin{enumerate}[(i)]
  \item global existence: $\omega=\infty$;
  \item $v_i$ and $u_i$ are bounded for all $i=1,\ldots,N$, independent of~$i$ and~$N$;
  \item there exist $\eps_1, \eps_2>0$, independent of~$i$ and~$N$, so that for all $i=1,\ldots,N$ and all $t\ge  0$ we have
    \begin{align*}
        &-M + \eps_1 \le \xi_i(t) \le -\eps_1\quad \text{and} \quad |w_i(t)| \le \psi(t) - \eps_2;
    \end{align*}
  \item for all $i=1,\ldots,N$ we have
  \[
    \hspace*{-3mm} \|v_i\|_\infty \le \frac{M}{\lambda} + \frac{1}{\lambda k_2} \left(\frac{\|\psi\|_\infty}{\eps_2} + \bar d\right)  + \left(\frac{k_1}{k_1+ \lambda k_2}\right)^i \|v_0\|_\infty.
  \]
  \vskip2mm%
\end{enumerate}
\end{theorem}

\begin{proof} Let $k_1>0$, $\lambda>0$, $d_{\max}>d_{\min}>0$ and $\psi\in\Psi$ be arbitrary control parameters and set $M:=d_{\max}-d_{\min}$. The parameter $k_2>0$ will be specified later. The proof is divided into several steps.

\textit{Step 1}: We show existence of a maximal solution. Define the set $\cD$ as in~\eqref{eq:cD} and observe that it is relatively open in $\R_{\ge 0}\times \R^N\times\R^N$ and contains the point $(0,x(0),v(0))$ by Assumption~\ref{Ass2}. Roughly speaking, the set $\cD$ is the intersection of all performance funnels associated with any two consecutive vehicles in the platoon. Further define the function $f:\cD\to\R^{2N}$ by
\begin{align*}
    &f(t,x,v) =\\
    & \begin{pmatrix} v_1 \\ \vdots \\ v_N \\ K(t,x_1-x_0(t),v_1,v_0(t)) - f_1(t,x_1,v_1) + d_1(t) \\
    K(t,x_2-x_1,v_2,v_1) - f_2(t,x_2,v_2) + d_2(t) \\ \vdots \\
    K(t,x_N-x_{N-1},v_N,v_{N-1}) - f_N(t,x_N,v_N) + d_N(t)
    \end{pmatrix},
\end{align*}
where~$K$ is defined in~\eqref{eq:def-K}. The closed-loop system~\eqref{eq:Sys},~\eqref{eq:SSFCC} is then equivalent to
\[
    \begin{pmatrix} \dot x(t)\\ \dot v(t) \end{pmatrix} = f(t,x(t),v(t))
\]
with initial condition~\eqref{eq:IC}. Since $f_1,\ldots,f_N$ are continuous in~$(t,x,v)$ by Assumption~\ref{Ass1} when $F_{i,r}$ is chosen as in~\eqref{eq:friction}, and $d_i\in L^\infty(\R_{\ge 0},\R)$, it follows that~$f$ is measurable and locally integrable in~$t$ and continuous in~$(x,v)$. Therefore, it follows from the theory of ordinary differential equations, see~\cite[\S\,10,~Thm.\,XX]{Walt98}, that there exists a solution, which can be extended to a maximal solution $(x,v) :[0,\omega)\to\R^{2N}$, $\omega\in(0,\infty]$. Furthermore, by maximality, the closure of the graph of $(x,v)$ is not a compact subset of~$\cD$.

\textit{Step 2}: We show that for
\[
    \eps_1 := \left(\|\psi\|_\infty + \tfrac{1}{\delta}\right)^{-1} < \delta,
\]
which is independent of~$i$ and~$N$, and where~$\delta$ is from Assumption~\ref{Ass2}, we have that $-M+\eps_1\le \xi_i(t)\le -\eps_1$ for all $t\in [0,\omega)$ and all $i=1,\ldots,N$. We show the first inequality and, seeking a contradiction, assume that there exist $i\in\{1,\ldots,N\}$ and $t_1\in [0,\omega)$ such that $\xi_i(t_1) < -M +\eps_1$. By Assumption~\ref{Ass2} we find that $t_1>0$ and hence
\[
    t_0 := \max \setdef{t\in[0,t_1)}{\xi_i(t) = -M +\eps_1}
\]
is well-defined.
Then $\xi_i(t) \le -M +\eps_1 < -M + \delta \stackrel{\text{Ass.~\ref{Ass2}}}{\le} -\delta$ and, by Step~1, $|w_i(t)|<\psi(t)$ for all $t\in[t_0,t_1]$ and we may compute that
\begin{align*}
  \dot \xi_i(t) &= v_i(t) - v_{i-1}(t) \stackrel{\eqref{eq:SSFCC}}{=} w_i(t) + \frac{1}{\xi_i(t)} + \frac{1}{M + \xi_i(t)}\\
  &\ge - \|\psi\|_\infty - \frac{1}{\delta} + \frac{1}{\eps_1} = 0
\end{align*}
for all $t\in [t_0,t_1]$, hence we have $\xi_i(t_0) \le \xi_i(t_1)$ and arrive at the contradiction
\[
    -M +\eps_1 = \xi_i(t_0) \le \xi_i(t_1) < -M +\eps_1.
\]
The proof of $\xi_i(t)\le -\eps_1$ for all $t\in [0,\omega)$ and all $i=1,\ldots,N$ is analogous and omitted.

\textit{Step 3}: We show that there exists $\eps_2>0$, independent of~$i$ and~$N$, such that $|w_i(t)|\le \psi(t) - \eps_2$ for all $t\in[0,\omega)$ and all $i=1,\ldots,N$.

\textit{Step 3a}: First, we show that there exists $\hat\eps_1 = \hat \eps_1(k_2)$ such that $|w_1(t)|\le \psi(t) - \hat\eps_1$ for all $t\in[0,\omega)$. Define $\theta := \inf_{t\ge 0} \psi(t)$ and choose    $\hat \eps_1 \le \min\left\{\delta, \frac{\theta}{2}\right\}$; $\hat\eps_1$ will be chosen even smaller later on. Seeking a contradiction, assume that there exists $t_1\in [0,\omega)$ such that $|w_1(t_1)| > \psi(t_1) - \hat\eps_1$. By Assumption~\ref{Ass2} we find that $t_1>0$ and hence
\[
    t_0 := \max \setdef{t\in[0,t_1)}{|w_1(t)| = \psi(t) - \hat\eps_1}
\]
is well-defined. Then
\[
    |w_1(t)| \ge \psi(t) - \hat\eps_1 \ge \frac{\theta}{2} \ \ \text{and}\ \ k_{1,3}(t) = \frac{1}{\psi(t) - |w_1(t)|} \ge \frac{1}{\hat\eps_1}
\]
for all $t\in[t_0,t_1]$ and we record that by~\eqref{eq:SSFCC}
\[
    \dot \xi_1(t) = v_1(t) - v_0(t) = w_1(t) + \frac{1}{\xi_1(t)} + \frac{1}{M+\xi_1(t)}
\]
and
\begin{equation}\label{eq:v1=}
    v_1(t) = v_0(t) + w_1(t) + \frac{1}{\xi_1(t)} + \frac{1}{M+\xi_1(t)},
\end{equation}
which we will frequently use in the following. We may now compute that
\begin{align*}
  & \tfrac12\ddt w_1(t)^2 = w_1(t)\left(\dot v_1(t) \!-\! \dot v_0(t) \!+\! \dot \xi_1(t) \left(\tfrac{1}{\xi_1(t)^2} \!+\! \tfrac{1}{(M+\xi_1(t))^2}\right)\!\right) \\
  &= w_1(t)\left(-\tfrac{k_1}{m_1}(v_1(t)- v_0(t)) - \tfrac{k_2}{m_1} (\xi_1(t) + \lambda v_1(t))\right. \\
  &\quad  - \tfrac{k_{1,3}(t)}{m_1} w_1(t) + \tfrac{1}{m_1}\big(d_1(t) - f_1(t,x_1(t),v_1(t))\big) - \dot v_0(t) \\
  &\quad \left.+ \left(w_1(t) + \tfrac{1}{\xi_1(t)} + \tfrac{1}{M+\xi_1(t)}\right) \left(\tfrac{1}{\xi_1(t)^2} + \tfrac{1}{(M+\xi_1(t))^2}\right)\right)\\
  &\stackrel{\eqref{eq:v1=},\,\text{Step~2}}{\le} -\left(\tfrac{k_1}{m_1} + \tfrac{\lambda k_2}{m_1} + \tfrac{1}{m_1 \hat\eps_1} - \tfrac{2}{\eps_1^2}\right) w_1(t)^2 \\
  &\quad + |w_1(t)| \left( \left(\tfrac{k_1}{m_1} + \tfrac{\lambda k_2}{m_1}\right) \frac{2}{\eps_1} + \tfrac{k_2}{m_1} M  + \tfrac{\lambda k_2}{m_1} |v_0(t)| \right.\\
  &\quad \left.+ |\dot v_0(t)| + \tfrac{\bar d}{m_1} + \tfrac{\bar \rho}{m_1} \left(\|\psi\|_\infty + \tfrac{2}{\eps_1} + |v_0(t)|\right)^2 + \tfrac{4}{\eps_1^3}\right)
\end{align*}
for all $t\in[t_0,t_1]$. Next, we choose $k_2$ sufficiently large such that
\begin{equation}\label{eq:k1k2suff1}
    k_1 + \lambda k_2 \ge \frac{2m_1}{\eps_1^2}.
\end{equation}
Since $v_0$ and $\dot v_0$ are bounded by Assumption~\ref{Ass2}, and invoking $|w_1(t)| \ge \theta/2$, there exist constants $c_1, c_2>0$, independent of~$k_2$, such that
\[
  \forall\, t\in [t_0,t_1]:\ \frac{\ddt w_1(t)^2}{2 |w_1(t)|} \le- \frac{\theta}{2m_1 \hat\eps_1} + k_2 c_1 + c_2
\]
and we may choose
\begin{equation}\label{eq:hat-eps1}
    \hat \eps_1 \le \frac{2m_1}{\theta} \left( k_2 c_1 + c_2 + \|\dot \psi\|_\infty\right)^{-1}
\end{equation}
so that
\begin{align*}
    |w_1(t_1)| - |w_1(t_0)| &= \int_{t_0}^{t_1} \frac{\ddt w_1(t)^2}{2 |w_1(t)|} \ds{t}\\
     &\le - \int_{t_0}^{t_1} \|\dot \psi\|_\infty \ds{t} \le \psi(t_1) - \psi(t_0),
\end{align*}
which leads to the contradiction
\[
    \hat \eps_1 = |w_1(t_0)| - \psi(t_0) \ge |w_1(t_1)| - \psi(t_1) > \hat \eps_1.
\]

\textit{Step 3b}: Next, we show that there exists $\hat\eps_i = \hat \eps_i(k_2,\hat \eps_{i-1}) \le \hat \eps_{i-1}$ such that $|w_i(t)|\le \psi(t) - \hat\eps_i$ for all $t\in[0,\omega)$ and all $i=2,\ldots,N$. By way of induction, observe that this assertion is true for $i=1$ by Step~4a and, fixing $i\in\{2,\ldots,N\}$, assume that it is true for $j=2,\ldots,i-1$. In particular $\hat \eps_{i-1} \le \hat \eps_{i-2} \le \ldots \le \hat \eps_1$ and $|w_{j}(t)|\le \psi(t) - \hat\eps_{i-1}$ for all $t\in[0,\omega)$ and all $j=1,\ldots,i-1$, thus for all $t\in [0,\omega)$ and all $j=1,\ldots,i-1$ we have
\begin{align*}
	\tfrac{m_j}{2} \ddt v_j(t)^2
	& = v_j(t)\big(u_j(t) - f_j(t,x_j(t),v_j(t)) + d_j(t)\big)\\
	&\stackrel{\mathclap{\text{Ass.~\ref{Ass1}}}}{\le} ~\, -k_1 (v_j(t) - v_{j-1}(t)) v_j(t) - k_2 e_j(t) v_j(t) \\
	&\qquad- k_{j,3}(t) w_j(t) v_j(t) + \bar d |v_j(t)|\\
	&\qquad - \tfrac12 \rho_j(t,x_j(t))\, C_{j,d}\, A_j |v_j(t)| v_j(t)^2\\
	&\le -(k_1+\lambda k_2) v_j(t)^2  + \big(\bar d + k_2 M \\
	&\qquad + k_{j,3}(t)\psi(t) + k_1 |v_{j-1}(t)|\big) |v_j(t)|\\
	&\le -(k_1+\lambda k_2) v_j(t)^2  + \big(\bar d + k_2 M \\
	&\qquad + \frac{\|\psi\|_\infty}{\hat\eps_j} + k_1 \|v_{j-1}\|_\infty\big) |v_j(t)|,
\end{align*}
where we have used $|\xi_j(t)| \le M$ (from Step~2), and hence it follows from a straightforward argument by contradiction that
\begin{align*}
     \forall\, t\in[0,\omega):\ |v_j(t)| & \le \max\Bigg\{|v_j(0)|, \\
     &\quad \frac{\bar d + k_2 M +  \tfrac{\|\psi\|_\infty}{\hat\eps_j}  + k_1  \|v_{j-1}\|_\infty}{k_1+\lambda k_2}\Bigg\}.
\end{align*}
Note that from boundedness of $v_0$ and inductively applying the above estimate for $j=0,\ldots,i-2$ it follows that $v_{j-1}$ is indeed bounded on $[0,\omega)$.  Using that $|v_j(0)|\le\frac{M}{\lambda}$ by Assumption~\ref{Ass2}, it then follows from a straightforward induction argument that
\begin{multline}
   \|v_{i-1}\|_\infty \le \frac{M}{\lambda} + \frac{\bar d}{\lambda k_2}
    + \frac{\|\psi\|_\infty}{\lambda k_2 \hat\eps_{i-1}}  + \left(\frac{k_1}{k_1+\lambda k_2}\right)^{i-1} \|v_0\|_\infty,\label{eq:est-v_i-2}
\end{multline}
where suprememum norms are taken on $[0,\omega)$.
 Now, choose
\[
    \hat \eps_i \le \hat \eps_{i-1} \le \hat \eps_1 \le \min\left\{\delta, \tfrac{\theta}{2}\right\}
\]
and note that $\hat\eps_i$ will be chosen even smaller later on. Seeking a contradiction, assume that there exists $t_1\in [0,\omega)$ such that $|w_i(t_1)| > \psi(t_1) - \hat\eps_i$. By Assumption~\ref{Ass2} we find that $t_1>0$ and hence
\[
    t_0 := \max \setdef{t\in[0,t_1)}{|w_i(t)| = \psi(t) - \hat\eps_i}
\]
is well-defined. Then
\[
    |w_i(t)| \ge \psi(t) - \hat\eps_i \quad \text{and}\quad k_{i,3}(t) = \frac{1}{\psi(t) - |w_i(t)|} \ge \frac{1}{\hat\eps_i}
\]
for all $t\in[t_0,t_1]$ and we record that by~\eqref{eq:SSFCC}
\[
    \dot \xi_i(t) = v_i(t) - v_{i-1}(t) = w_i(t) + \frac{1}{\xi_i(t)} + \frac{1}{M+\xi_i(t)},
\]
which we will frequently use in the following. We may now compute, similar to Step~3a, that
\begin{align*}
  & \tfrac12\ddt w_i(t)^2 
  \stackrel{\text{Step~2}}{\le} -\left(\tfrac{k_1}{m_i} + \tfrac{\lambda k_2}{m_i} + \tfrac{1}{m_i \hat\eps_i} - \tfrac{2}{\eps_1^2}\right) w_i(t)^2 \\
  &\quad + |w_i(t)| \left( \left(\tfrac{k_1}{m_i} + \tfrac{\lambda k_2}{m_i}\right) \tfrac{2}{\eps_1} + \tfrac{k_1}{m_{i-1}}\left( \|\psi\|_\infty +\tfrac{2}{\eps_1}\right) \right.\\
  &\quad  + k_2 M  \left(\tfrac{1}{m_i} +\tfrac{1}{m_{i-1}}\right) + \lambda k_2 |v_{i-1}(t)|\, \left|\tfrac{1}{m_i} -\tfrac{1}{m_{i-1}}\right| \\
  &\quad + \tfrac{\psi(t)}{m_{i-1}\hat \eps_{i-1}}   + 2\bar d \left(\tfrac{1}{m_i} +\tfrac{1}{m_{i-1}}\right) \\
  &\quad\left. + \tfrac{\bar \rho}{m_i} \left(\|\psi\|_\infty + \tfrac{2}{\eps_1} + |v_{i-1}(t)|\right)^2 + \tfrac{\bar \rho}{m_{i-1}} |v_{i-1}(t)|^2 + \tfrac{4}{\eps_1^3}\right)
\end{align*}
for all $t\in[t_0,t_1]$. By Assumptions~\ref{Ass1} and~\ref{Ass3} we have that
\[
   1 \le \frac{\bar m}{m_i}\quad\text{and}\quad \frac{1}{m_{i-1}} \le \max \left\{\frac{q}{m_i}, \frac{\bar m}{m_i \, \min_{j=1,\ldots,N_0} m_j}\right\}.
\]
Invoking~\eqref{eq:k1k2suff1} and~\eqref{eq:est-v_i-2} and boundedness of~$v_0$ it follows that there exist constants $c_1, c_2, c_3, c_4 > 0$, independent of~$i$,~$N$ and~$k_2$, such that
\begin{multline}
  \hspace*{-3mm}  \forall\, t\in [t_0,t_1]:\ \frac{\ddt w_i(t)^2}{2 |w_i(t)|} \le - \frac{\psi(t)-\hat\eps_i}{m_i \hat\eps_i} + \frac{\psi(t)}{m_{i-1}\hat \eps_{i-1}} \\
    + \frac{\psi(t)}{\hat\eps_{i-1}} \left|\frac{1}{m_i} -\frac{1}{m_{i-1}}\right| + \frac{c_1}{m_i k_2\hat \eps_{i-1}} + \frac{c_2}{m_i k_2^2 \hat \eps_{i-1}^2}  + \frac{k_2}{m_i} c_3 + \frac{c_4}{m_i}.\label{eq:est-wi}
\end{multline}
Clearly, we may now choose $\hat \eps_i$ small enough so that $\frac{\ddt w_i(t)^2}{2 |w_i(t)|} \le -\|\dot \psi\|_\infty$ which, similar to Step~3a, will lead to a contradiction, thus proving the assertion.

\textit{Step 3c}: We show that for sufficiently large $k_2>0$ there exists~$\eps_2>0$ such that the sequence $(\hat\eps_i)$ from Step~3b can be chosen in such a way that it is uniformly bounded from below by~$\eps_2$, independent of~$i$ and~$N$. In virtue of Assumption~\ref{Ass3}, we restrict ourselves to the case $N\ge N_0$ and $i\ge N_0$. We now return to equation~\eqref{eq:est-wi} and instead of choosing $\hat \eps_i$ small enough, we choose it in a specific way. To this end, observe that by Assumption~\ref{Ass3} we have
\[
    \left|\frac{1}{m_i} -\frac{1}{m_{i-1}}\right| \le \frac{p}{m_{i-1}} \le \frac{pq}{m_{i}}
\]
and hence
\begin{multline*}
     \frac{\ddt w_i(t)^2}{2 |w_i(t)|} \le - \frac{\psi(t)}{m_i \hat\eps_i} + \frac{\psi(t)}{m_{i}\hat \eps_{i-1}} (1+p) q + \frac{c_1}{m_i k_2\hat \eps_{i-1}}\\
      + \frac{c_2}{m_i k_2^2 \hat \eps_{i-1}^2}
     + \frac{k_2}{m_i} c_3 + \frac{c_4+1}{m_i}
\end{multline*}
for all $t\in[t_0,t_1]$. The problem now boils down to finding $\hat \eps_i \le \hat \eps_{i-1}$ such that the right hand side of the above equation becomes less than $-\|\dot \psi\|_\infty$ and, at the same time, $\hat \eps_i$ is uniformly bounded away from zero for $i\ge N_0$, when $k_2$ is chosen sufficiently large. With the new constants
\begin{align*}
  \tilde c_1 &:= c_1/\theta, & \tilde c_2&:= c_2/\theta,\\
  \tilde c_3 &:= c_3/\theta, & \tilde c_4&:= (c_4 + 1 + \bar m \|\psi\|_\infty)/\theta,
\end{align*}
and $z_i := 1/\hat \eps_i$ we can achieve this by defining the sequence $(z_i)$ as
\[
    z_i = (1+p)q z_{i-1} + \frac{\tilde c_1}{k_2} z_{i-1} + \frac{\tilde c_2}{k_2^2} z_{i-1}^2
    + k_2 \tilde c_3 + \tilde c_4
\]
for $i\ge N_0$. Set $\alpha:=  (1+p)q + \frac{\tilde c_1}{k_2}$ and, in a first step, choose $k_2$ large enough such that $\alpha < 1$, which is possible since $(1+p)q < 1$ by Assumption~\ref{Ass3}. Now choose $k_2$ large enough such that
\[
    \frac{k_2^2(1-\alpha)^2}{4\tilde c_2} \ge k_2 \tilde c_3 + \tilde c_4
\]
and define
\[
    \hat z := \frac{k_2^2(1-\alpha)}{2\tilde c_2} + \sqrt{\frac{k_2^4(1-\alpha)^2}{4\tilde c_2^2} - \frac{k_2^2}{\tilde c_2}\big( k_2 \tilde c_3 + \tilde c_4\big)}.
\]
Observe that
\[
    \forall\, i\ge N_0:\ z_{i-1} \le \hat z\quad \implies \quad z_i \le \hat z.
\]
Since the sequence $(z_i)$ is monotonically increasing, this means that it is bounded by $\hat z$ provided that $z_1 \le \hat z$. If $z_1 = 1/\hat \eps_1$ is chosen with equality in~\eqref{eq:hat-eps1}, then it is proportional to~$k_2$, while $\hat z$ is proportional to~$k_2^2$. Consequently, $z_1 \le \hat z$ will be true for~$k_2$ sufficiently large. Finally, the constant $\eps_2>0$ may be defined as
\[
    \eps_2 := \frac{1}{\hat z} > 0.
\]

\textit{Step~4}: We show that $\omega = \infty$, and hence assertion~(i). Assuming that $\omega < \infty$ and taking Steps~2 and~3 into account, it follows that the graph of the solution~$(x,v)$ is a compact subset of~$\cD$, which contradicts the findings of Step~1.

\textit{Step 5}: We show assertions~(iii) and~(iv). Assertion~(iii) follows from Step~2 and Step~3. Assertion~(iv) is a direct consequence of~\eqref{eq:est-v_i-2} and Step~3.

\textit{Step 6}: We show assertion~(ii). Clearly, the boundedness of~$v_i$, independent of~$i$ and~$N$, follows from Step~5 and boundedness of~$v_0$ by Assumption~\ref{Ass2}. Since $\xi_i$ is bounded by~$M$ and~$k_{i,3}$ is bounded by $1/\eps_2$, independent of~$i$ and~$N$, the uniform boundedness of~$u_i$ may be inferred. This finishes the proof.
\end{proof}

\begin{remark}
In view of Theorem~\ref{thm:main} the decentralized nature of the controller~\eqref{eq:SSFCC} may be questioned, because the parameter $k_2 >0$ must be chosen sufficiently large and all vehicles~$i$ must agree on its value. One possibility is to choose a very large value encompassing any possible vehicle configuration, thus keeping the decentralized nature of the controller. If a more realistic viewpoint is taken, then some communication capabilities are required, at least between neighboring vehicles, so that a brief handshake is possible when new vehicles join the platoon. As long as those vehicles also satisfy Assumptions~\ref{Ass1}--\ref{Ass3} without changing the parameters therein, it suffices to communicate the values of $k_1, k_2$ (and maybe the function~$\psi$, if it is used as a flexible parameter) to the joining vehicles and the control~\eqref{eq:SSFCC} will still be feasible. This is based on the fact that, although~$k_2$ must be sufficiently large, it is independent of the number of vehicles in the platoon.
\end{remark}

\begin{remark}
  We like to provide some guidelines for choosing the parameters~$k_1$ and~$k_2$. Although an inspection of the proof of Theorem~3.4 provides a lower bound for~$k_2$, this is quite conservative and not suitable for use in practice. Therefore, we conducted a number of simulations and it turns out that choosing $k_1 = k_2 = 2 m_{\max}$, where $m_{\max}$ is the maximal mass of a vehicle in the platoon, serves as a good rule of thumb.  This means that, roughly speaking, the velocity and acceleration signals do not exhibit any peaks, the velocity plateau is not too far away and the control effort is comparatively small.
\end{remark}

\section{Simulations}\label{Sec:Sim}

We illustrate the novel controller~\eqref{eq:SSFCC} by an application to a platoon of twenty inhomogeneous vehicles with dynamics~\eqref{eq:Sys}, which follow a leader vehicle with position~$x_0$, velocity~$v_0$ and acceleration~$a_0$. We consider two different scenarios. The first scenario illustrates that safety is guaranteed even in the case of a sudden full brake of the leader vehicle (i.e., with maximal deceleration of $\SI{-5}{\metre\per\second^2}$). In the second scenario the leader vehicle follows a vivid curve with a strongly varying acceleration. Both scenarios are of an academic nature and serve the purpose of demonstrating the power of the controller design~\eqref{eq:SSFCC}: Even under extreme braking (Scenario~1) or strongly varying acceleration (Scenario~2) of the leader, the controller is able to achieve the objectives (O1)--(O4). Even more so, again to illustrate the power of the controller design in these situations, we chose a very tight range $d_{\min} = 2 < 15 = d_{\max}$ for the inter-vehicle distances, yet the simulations verify a very good controller performance.

For the simulations we choose all parameters of the vehicles in~\eqref{eq:Sys}, apart from their masses, to be equal with typical values (taken from~\cite{AstrMurr08}) summarized in Table~\ref{Tab:Param}.

\begin{table}[h!tb]
\centering
\begin{tabular}{|c|c|c|c|c|c|}
  \hline
    $m_i$ & $\theta_i(x)$ & $\rho_i(t,x)$ & $C_{i,d}$ & $C_{i,r}$ & $A_i$  \\
   \hline
    $1500\!+\!(-1)^i 300\,\si{\kilo\gram}$ & $\SI{0}{\radian}$ 
     & $\SI{1.3}{\kilo\gram\per\metre^3}$ & 0.32 & 0.01 & $\SI{2.4}{\metre^2}$\\
  \hline
\end{tabular}
\caption{Parameter values for the vehicles in the platoon with dynamics~\eqref{eq:Sys}.}
\label{Tab:Param}
\end{table}

For the approximated friction model~\eqref{eq:friction} we choose the parameter $\alpha = 100$. The initial conditions~\eqref{eq:IC} are chosen as $x_i^0=x_{i-1}^0-11\,\si{\metre}$, $x_0^0 = x_0(0) = 0 \,\si{\metre}$ and $v_i^0 = v_0(0) = \SI{20}{\metre\per\second}$ for all $i=1,\ldots,20$. The controller design parameters~\eqref{eq:FC-param} are chosen as $M=d_{\max}-d_{\min} = 13$, $\lambda = 0.5$, $k_1 = k_2 = 3600$ and $\psi(t) = e^{-2t} + 1$ for $t\ge 0$. All simulations have been performed in MATLAB (solver: \textsc{ode15s}, rel.\ tol.: $10^{-10}$, abs.\ tol.: $10^{-10}$) over the time interval 0--$40\,\si{\second}$.

\textbf{Scenario 1:} The leader position~$x_0$ and velocity~$v_0$ are chosen so that after a period of safe following the leader vehicle suddenly fully brakes with maximal deceleration of $\SI{-5}{\metre\per\second^2}$.

\captionsetup[subfloat]{labelformat=empty}
\begin{figure}[h!tb]
  \centering
  \subfloat[\qquad\quad Fig.~\ref{fig:sim1}a: Inter-vehicle distances, safety distance $d_{\min}$ and maximal\newline \phantom{mmk}\quad distance~$d_{\max}$]
{
\centering
\hspace{-5mm}
  \includegraphics[width=0.53\textwidth]{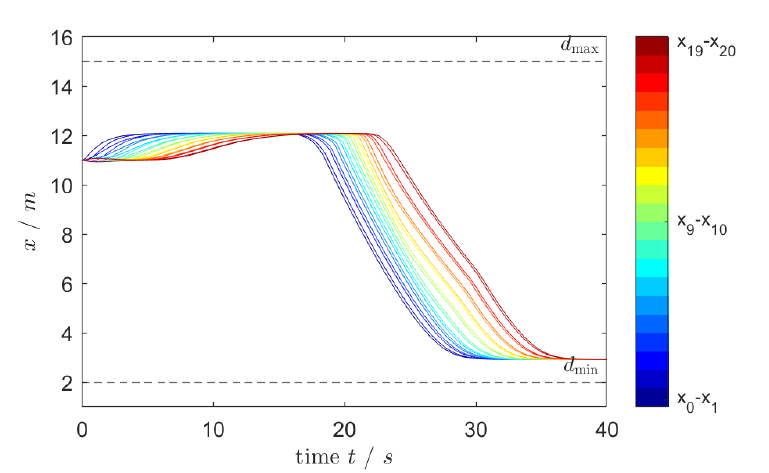}
\label{fig:sim2-e}
}\\
\subfloat[Fig.~\ref{fig:sim1}b: Velocities]
{
\centering
\hspace{-5mm}
  \includegraphics[width=0.53\textwidth]{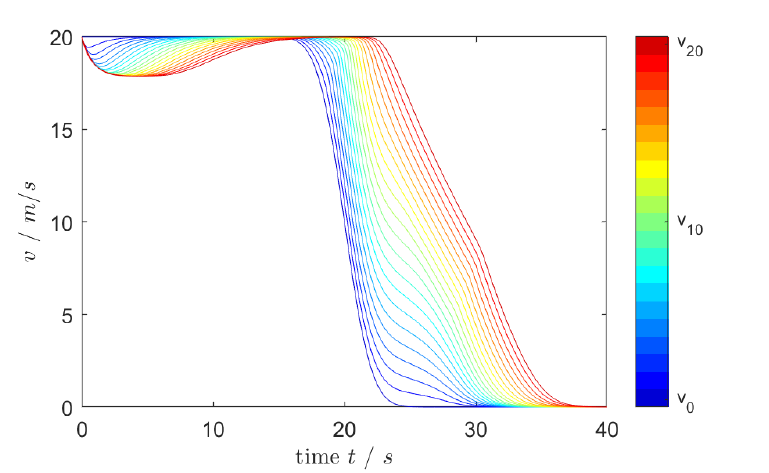}
\label{fig:sim2-u}
}\\
\subfloat[Fig.~\ref{fig:sim1}c: Accelerations]
{
\centering
\hspace{-5mm}
  \includegraphics[width=0.53\textwidth]{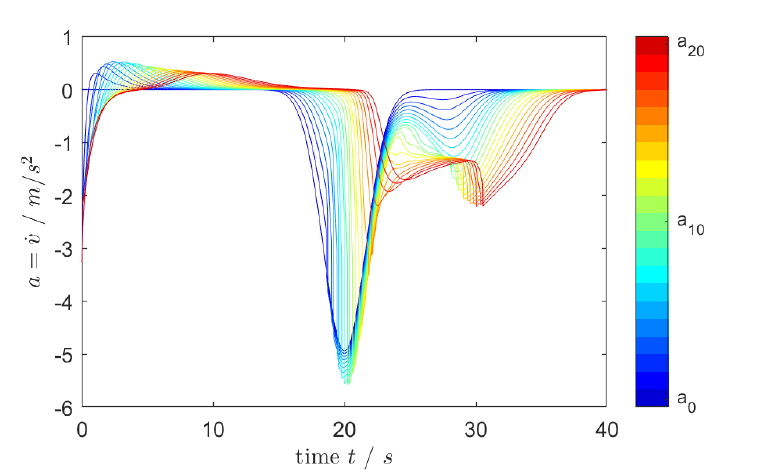}
\label{fig:sim2-u}
}
\caption{Simulation, under controller~\eqref{eq:SSFCC}, of system~\eqref{eq:Sys} with 20 vehicles following a leader in Scenario~1 and parameters as in Table~\ref{Tab:Param}.}
\label{fig:sim1}
\end{figure}

The simulation of the platoon~\eqref{eq:Sys} under the controller~\eqref{eq:SSFCC} in this scenario is depicted in Fig.~\ref{fig:sim1}. It can be seen in Fig.~\ref{fig:sim1}a that the prescribed safety and maximal distances are always guaranteed. From Fig.~\ref{fig:sim1}b  we can observe that the velocities reach a plateau (in the time interval 0--$10\,\si{\second}$), thus the controlled platoon exhibits practical velocity string stability. The accelerations of the vehicles, depicted in Fig.~\ref{fig:sim1}c, are similar to that of the leader for the first few followers and become much smaller with increasing vehicle index.

\textbf{Scenario 2:} The leader position trajectory is chosen as $x_0(t) = 10 + 19t-10 \cos(t/5)+\tfrac12 \sin(2t)$ so that the leader exhibits a strongly varying acceleration.

\captionsetup[subfloat]{labelformat=empty}
\begin{figure}[h!tb]
  \centering
  \subfloat[\qquad\quad Fig.~\ref{fig:sim2}a: Inter-vehicle distances, safety distance $d_{\min}$ and maximal\newline \phantom{mmk}\quad distance~$d_{\max}$]
{
\centering
\hspace{-5mm}
  \includegraphics[width=0.53\textwidth]{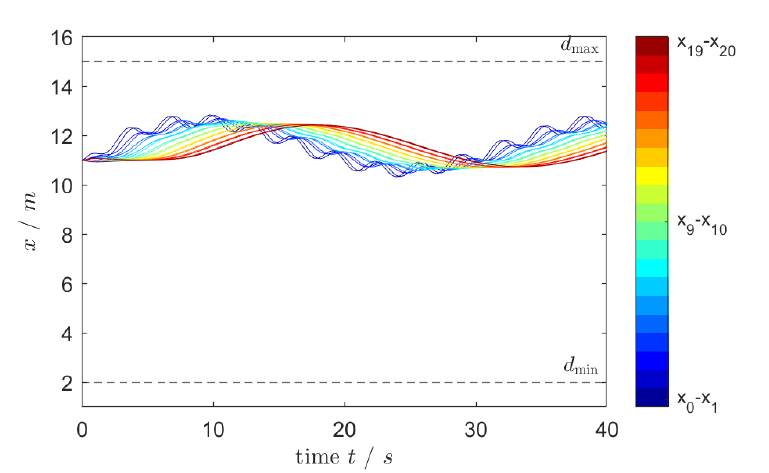}
\label{fig:sim2-e}
}\\
\subfloat[Fig.~\ref{fig:sim2}b: Velocities]
{
\centering
\hspace{-5mm}
  \includegraphics[width=0.53\textwidth]{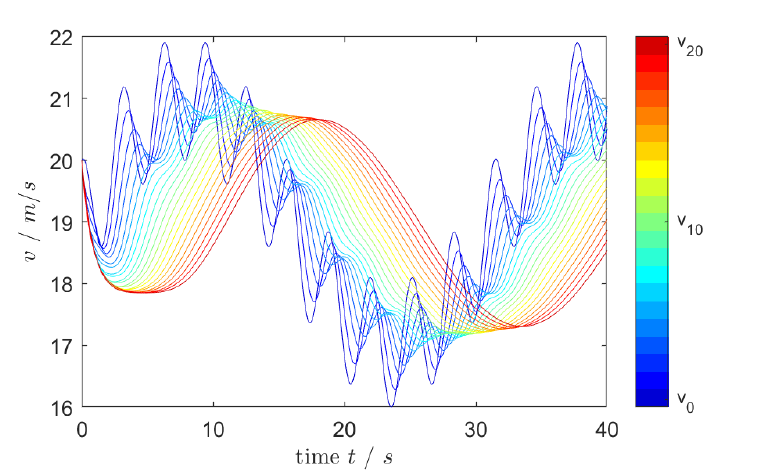}
\label{fig:sim2-u}
}\\
\subfloat[Fig.~\ref{fig:sim2}c: Accelerations]
{
\centering
\hspace{-5mm}
  \includegraphics[width=0.53\textwidth]{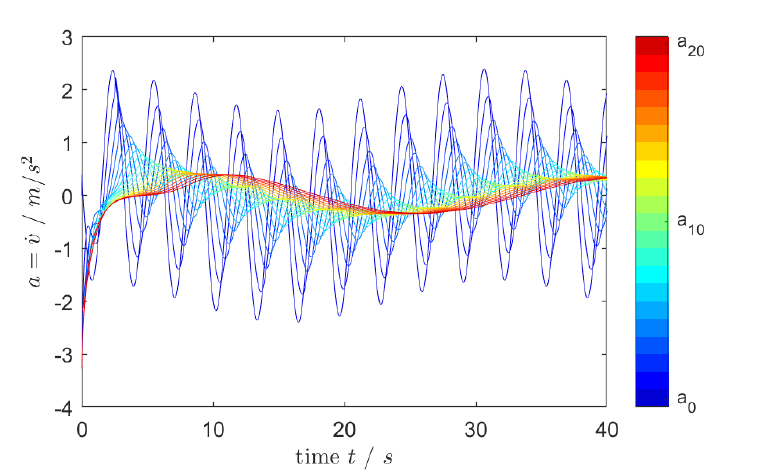}
\label{fig:sim2-u}
}
\caption{Simulation, under controller~\eqref{eq:SSFCC}, of system~\eqref{eq:Sys} with 20 vehicles following a leader in Scenario~2 and parameters as in Table~\ref{Tab:Param}.}
\label{fig:sim2}
\end{figure}

The simulation of the platoon~\eqref{eq:Sys} under the controller~\eqref{eq:SSFCC} in this scenario is shown in Fig.~\ref{fig:sim2}. Again, as depicted in Fig.~\ref{fig:sim2}a, the inter-vehicle distances stay within the corridor given by the prescribed safety and maximal distances, and even vary by only $\SI{2}{\metre}$. The velocities and accelerations of the followers are shown in Figs.~\ref{fig:sim2}b and~\ref{fig:sim2}c and it can be seen that the strongly varying behavior of the leader is ``smoothed out'' with increasing vehicle index. This is another characteristic of the practical velocity string stability of the controlled platoon.

Overall, the simulations show that via the novel controller~\eqref{eq:SSFCC} it can be achieved that all inter-vehicles distances stay within a prescribed tight corridor, guaranteeing safety and allowing for a very good traffic flow. At the same time, practical velocity string stability is achieved under the control and variations of the leader velocity are smoothed out with increasing vehicle index. All of this is guaranteed even in extreme scenarios for platoons of inhomogeneous vehicles. Furthermore, the controller can be implemented in a decentralized fashion, requiring only local information.

\section{Conclusion}\label{Sec:Concl}

We presented a new decentralized control design for vehicle platoons and have proved that it is able to both guarantee safety and a good traffic flow, and at the same time practical velocity string stability is achieved. The results were illustrated by the simulation of different scenarios. Future research should focus on the relaxation of Assumption~\ref{Ass3}, e.g.\ by allowing the vehicles to have access  to some common information. Furthermore, the simulations exhibit a certain synchronization behavior of the vehicles in the platoon, when the leader velocity remains constant, and it is an open question, whether this can be proved. Earlier results obtained for multi-agent systems with relative degree one~\cite{LeeBerg23} suggest that this is a common phenomenon. Another open question concerns the incorporation of input-constraints, which are always present in real-world applications. It should be investigated under which conditions on the parameters the control is feasible under input constraints, or whether the present approach can be combined with recent funnel control designs as presented e.g.\ in~\cite{Berg24}.

\bibliographystyle{IEEEtran}

\begin{thebibliography}{10}
\providecommand{\url}[1]{#1}
\csname url@samestyle\endcsname
\providecommand{\newblock}{\relax}
\providecommand{\bibinfo}[2]{#2}
\providecommand{\BIBentrySTDinterwordspacing}{\spaceskip=0pt\relax}
\providecommand{\BIBentryALTinterwordstretchfactor}{4}
\providecommand{\BIBentryALTinterwordspacing}{\spaceskip=\fontdimen2\font plus
\BIBentryALTinterwordstretchfactor\fontdimen3\font minus
  \fontdimen4\font\relax}
\providecommand{\BIBforeignlanguage}[2]{{%
\expandafter\ifx\csname l@#1\endcsname\relax
\typeout{** WARNING: IEEEtran.bst: No hyphenation pattern has been}%
\typeout{** loaded for the language `#1'. Using the pattern for}%
\typeout{** the default language instead.}%
\else
\language=\csname l@#1\endcsname
\fi
#2}}
\providecommand{\BIBdecl}{\relax}
\BIBdecl

\bibitem{HedrTomi94}
J.~K. Hedrick, M.~Tomizuka, and P.~Varaiya, ``Control issues in automated
  highway systems,'' \emph{{IEEE} Control Syst. Mag.}, vol.~14, no.~6, pp.
  21--32, 1994.

\bibitem{BonnFrit00}
C.~Bonnet and H.~Fritz, ``Fuel consumption reduction in a platoon:
  {E}xperimental results with two electronically coupled trucks at close
  spacing,'' Tech. Rep., 2000, sAE International, SAE Technical Paper
  2000-01-3056.

\bibitem{LeviAtha66}
W.~Levine and M.~Athans, ``On the optimal error regulation of a string of
  moving vehicles,'' \emph{{IEEE} Trans. Autom. Control}, vol.~11, pp.
  355--361, 1966.

\bibitem{Pepp74}
L.~E. Peppard, ``String stability of relative-motion pid vehicle control
  systems,'' \emph{{IEEE} Trans. Autom. Control}, vol.~19, no.~5, pp. 579--581,
  1974.

\bibitem{varaiya_1993}
P.~Varaiya, ``Smart cars on smart roads: problems of control,'' \emph{{IEEE}
  Trans. Autom. Control}, vol.~38, no.~2, pp. 195--207, 1993.

\bibitem{IoanChie93}
P.~A. Ioannou and C.~C. Chien, ``Autonomous intelligent cruise control,''
  \emph{{IEEE} Trans. Veh. Technol.}, vol.~42, no.~4, pp. 657--672, 1993.

\bibitem{SwarHedr99}
D.~Swaroop and J.~K. Hedrick, ``Constant spacing strategies for platooning in
  automated highway systems,'' \emph{J. Dyn. Syst. Meas. Control}, vol. 121,
  no.~3, pp. 462--470, 1999.

\bibitem{AlamBess15}
A.~Alam, B.~Besselink, V.~Turri, J.~M{\aa}rtensson, and K.~H. Johansson,
  ``Heavy-duty vehicle platooning towards sustainable freight transportation:
  {A} cooperative method to enhance safety and efficiency,'' \emph{{IEEE}
  Control Syst. Mag.}, vol.~35, no.~6, pp. 34--56, 2015.

\bibitem{ploeg_2014b}
J.~Ploeg, D.~P. Shukla, N.~van~de Wouw, and H.~Nijmeijer, ``Controller
  synthesis for string stability of vehicle platoons,'' \emph{{IEEE} Trans.
  Intelligent Transp. Systems}, vol.~15, no.~2, pp. 854--865, 2014.

\bibitem{zheng_2016}
Y.~Zheng, S.~E. Li, J.~Wang, D.~Cao, and K.~Li, ``Stability and scalability of
  homogeneous vehicular platoon: {S}tudy on the influence of information flow
  topologies,'' \emph{{IEEE} Trans. Intelligent Transp. Systems}, vol.~17,
  no.~1, pp. 14--26, 2016.

\bibitem{guanetti_2018}
J.~Guanetti, Y.~Kim, and F.~Borrelli, ``Control of connected and automated
  vehicles: {S}tate of the art and future challenges,'' \emph{Annual Reviews in
  Control}, vol.~45, pp. 18--40, 2018.

\bibitem{ersal_2020}
T.~Ersal, I.~Kolmanovsky, N.~Masoud, N.~Ozay, J.~Scruggs, R.~Vasudevan, and
  G.~Orosz, ``Connected and automated road vehicles: state of the art and
  future challenges,'' \emph{Vehicle System Dynamics}, vol.~58, no.~5, pp.
  672--704, 2020.

\bibitem{Lunz92}
J.~Lunze, \emph{Feedback {C}ontrol of {L}arge {S}cale {S}ystems}, ser.
  International {S}eries in {S}ystems and {C}ontrol {E}ngineering.\hskip 1em
  plus 0.5em minus 0.4em\relax Hemel Hempstead: Prentice-Hall, 1992.

\bibitem{PeteMidd14}
A.~A. Peters, R.~H. Middleton, and O.~Mason, ``Leader tracking in homogeneous
  vehicle platoons with broadcast delays,'' \emph{Automatica}, vol.~50, no.~1,
  pp. 64--74, 2014.

\bibitem{SeilPant04}
P.~Seiler, A.~Pant, and K.~Hedrick, ``Disturbance propagation in vehicle
  strings,'' \emph{{IEEE} Trans. Autom. Control}, vol.~49, no.~10, pp.
  1835--1842, 2004.

\bibitem{dunbar_2006}
W.~B. Dunbar and R.~M. Murray, ``Distributed receding horizon control for
  multi-vehicle formation stabilization,'' \emph{Automatica}, vol.~42, no.~4,
  pp. 549--558, 2006.

\bibitem{CampJia02}
E.~Camponogara, D.~Jia, B.~H. Krogh, and S.~Talukdar, ``Distributed model
  predictive control,'' \emph{{IEEE} Control Syst. Mag.}, vol.~22, no.~1, pp.
  44--52, 2002.

\bibitem{SwarHedr96}
D.~Swaroop and J.~K. Hedrick, ``String stability of interconnected systems,''
  \emph{{IEEE} Trans. Autom. Control}, vol.~41, no.~3, pp. 349--357, 1994.

\bibitem{wijnbergen_2020}
P.~Wijnbergen and B.~Besselink, ``Existence of decentralized controllers for
  vehicle platoons: {O}n the role of spacing policies and available
  measurements,'' \emph{Systems {\&} Control Letters}, vol. 145, p. 104796,
  2020.

\bibitem{NausVugt10}
G.~J.~L. Naus, R.~P.~A. Vugts, J.~Ploeg, M.~J.~G. van~de Molengraft, and
  M.~Steinbuch, ``String-stable {CACC} design and experimental validation: {A}
  frequency-domain approach,'' \emph{{IEEE} Trans. Vehicular Technology},
  vol.~59, no.~9, pp. 4268--4279, 2010.

\bibitem{OencPloe14}
S.~\"Onc\"u, J.~Ploeg, N.~van~de Wouw, and H.~Nijmeijer, ``Cooperative adaptive
  cruise control: Network-aware analysis of string stability,'' \emph{{IEEE}
  Trans. Intelligent Transp. Systems}, vol.~15, no.~4, pp. 1527--1537, 2014.

\bibitem{BessJoha17}
B.~Besselink and K.~H. Johannson, ``String stability and a delay-based spacing
  policy for vehicle platoons subject to disturbances,'' \emph{{IEEE} Trans.
  Autom. Control}, vol.~62, no.~9, pp. 4376--4391, 2017.

\bibitem{besselink_2018}
B.~Besselink and S.~Knorn, ``Scalable input-to-state stability for performance
  analysis of large-scale networks,'' \emph{IEEE Control Systems Letters},
  vol.~2, no.~3, pp. 507--512, 2018.

\bibitem{Sont89a}
E.~D. Sontag, ``Smooth stabilization implies coprime factorization,''
  \emph{{IEEE} Trans. Autom. Control}, vol.~34, no.~4, pp. 435--443, 1989.

\bibitem{FengZhan19}
S.~Feng, Y.~Zhang, S.~E. Li, Z.~Cao, H.~X. Liu, and L.~Li, ``String stability
  for vehicular platoon control: Definitions and analysis methods,''
  \emph{Annu. Rev. Control}, 2019.

\bibitem{IlchRyan02b}
A.~Ilchmann, E.~P. Ryan, and C.~J. Sangwin, ``Tracking with prescribed
  transient behaviour,'' \emph{{ESAIM} Control Optim. Calc. Var.}, vol.~7, pp.
  471--493, 2002.

\bibitem{BergIlch21}
T.~Berger, A.~Ilchmann, and E.~P. Ryan, ``Funnel control of nonlinear
  systems,'' \emph{Math. Control Signals Syst.}, vol.~33, pp. 151--194, 2021.

\bibitem{Hack17}
C.~M. Hackl, \emph{Non-identifier Based Adaptive Control in
  Mechatronics--Theory and Application}, ser. Lecture Notes in Control and
  Information Sciences.\hskip 1em plus 0.5em minus 0.4em\relax Cham,
  Switzerland: Springer-Verlag, 2017, vol. 466.

\bibitem{BergDrue21}
T.~Berger, S.~Drücker, L.~Lanza, T.~Reis, and R.~Seifried, ``Tracking control
  for underactuated non-minimum phase multibody systems,'' \emph{Nonlinear
  Dynamics}, vol. 104, pp. 3671--3699, 2021.

\bibitem{DrueLanz24}
S.~Drücker, L.~Lanza, T.~Berger, T.~Reis, and R.~Seifried, ``Experimental
  validation for the combination of funnel control with a feedforward control
  strategy,'' \emph{Multibody System Dynamics}, 2024.

\bibitem{BergReis14a}
T.~Berger and T.~Reis, ``Zero dynamics and funnel control for linear electrical
  circuits,'' \emph{J. Franklin Inst.}, vol. 351, no.~11, pp. 5099--5132, 2014.

\bibitem{SenfPaug14}
A.~Senfelds and A.~Paugurs, ``Electrical drive {DC} link power flow control
  with adaptive approach,'' in \emph{Proc. 55th Int. Sci. Conf. Power Electr.
  Engg.}, 2014, pp. 30--33.

\bibitem{PompWeye15}
A.~Pomprapa, S.~Weyer, S.~Leonhardt, M.~Walter, and B.~Misgeld, ``Periodic
  funnel-based control for peak inspiratory pressure,'' in \emph{Proc. 54th
  {IEEE} Conf. Decis. Control}, 2015, pp. 5617--5622.

\bibitem{BergPuch22}
T.~Berger, M.~Puche, and F.~L. Schwenninger, ``Funnel control for a moving
  water tank,'' \emph{Automatica}, vol. 135, p. Article 109999, 2022.

\bibitem{BergRaue18}
T.~Berger and A.-L. Rauert, ``A universal model-free and safe adaptive cruise
  control mechanism,'' in \emph{Proceedings of the MTNS 2018}, Hong Kong, 2018,
  pp. 925--932.

\bibitem{BergRaue20}
------, ``Funnel cruise control,'' \emph{Automatica}, vol. 119, p. Article
  109061, 2020.

\bibitem{verginis_2018}
C.~K. Verginis, C.~P. Bechlioulis, D.~V. Dimarogonas, and K.~J. Kyriakopoulos,
  ``Robust distributed control protocols for large vehicular platoons with
  prescribed transient and steady-state performance,'' \emph{IEEE Transactions
  on Control Systems Technology}, vol.~26, no.~1, pp. 299--304, 2018.

\bibitem{GuoLi19}
G.~Guo and D.~Li, ``Adaptive sliding mode control of vehicular platoons with
  prescribed tracking performance,'' \emph{IEEE Transactions on Vehicular
  Technology}, vol.~68, no.~8, pp. 7511--7520, 2019.

\bibitem{ArmsDupo94}
B.~Armstrong-H{\'e}louvry, P.~E. Dupont, and C.~Canudas-de Wit, ``A survey of
  models, analysis tools and compensation methods for the control of machines
  with friction,'' \emph{Automatica}, vol.~30, no.~7, pp. 1083--1138, 1994.

\bibitem{LeinNijm04}
R.~I. Leine and H.~Nijmeijer, \emph{Dynamics and bifurcations of non-smooth
  mechanical systems}, ser. Lecture notes in applied and computational
  mechanics.\hskip 1em plus 0.5em minus 0.4em\relax Berlin-Heidelberg:
  Springer-Verlag, 2004, vol.~18.

\bibitem{Ilch13}
A.~Ilchmann, ``Decentralized tracking of interconnected systems,'' in
  \emph{Mathematical System Theory - Festschrift in Honor of Uwe Helmke on the
  Occasion of his Sixtieth Birthday}, K.~H\"{u}per and J.~Trumpf, Eds.\hskip
  1em plus 0.5em minus 0.4em\relax CreateSpace, 2013, pp. 229--245.

\bibitem{Walt98}
W.~Walter, \emph{Ordinary Differential Equations}.\hskip 1em plus 0.5em minus
  0.4em\relax New York: Springer-Verlag, 1998.

\bibitem{AstrMurr08}
K.~J. {\AA}str{\"o}m and R.~M. Murray, \emph{Feedback Systems: An Introduction
  for Scientists and Engineers}.\hskip 1em plus 0.5em minus 0.4em\relax
  Princeton, NJ: Princeton University Press, 2008.

\bibitem{LeeBerg23}
J.~G. Lee, T.~Berger, S.~Trenn, and H.~Shim, ``Edge-wise funnel output
  synchronization of heterogeneous agents with relative degree one,''
  \emph{Automatica}, vol. 156, p. Article 111204, 2023.

\bibitem{Berg24}
T.~Berger, ``Input-constrained funnel control of nonlinear systems,''
  \emph{{IEEE} Trans. Autom. Control}, vol.~69, no.~8, pp. 5368--5382, 2024.

\end{thebibliography}



\begin{IEEEbiography}[{\includegraphics[width=1.05in,clip,keepaspectratio]{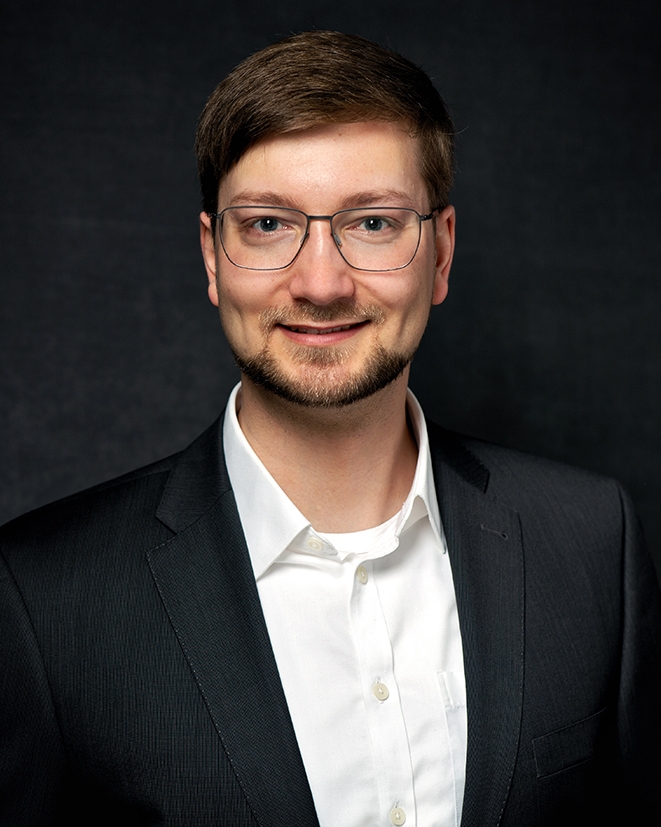}}]{Thomas Berger} was born in Germany in 1986. He received his B.Sc. (2008), M.Sc. (2010), and Ph.D. (2013), all in Mathematics and from Technische Universit\"at Ilmenau, Germany. From 2013 to 2018 Dr. Berger was a postdoctoral researcher at the Department of Mathematics, Universit\"at Hamburg, Germany. Since January 2019 he is a Juniorprofessor at the Institute for Mathematics, Universit\"at Paderborn, Germany. His research interest encompasses adaptive control, optimization-based control, differential–algebraic systems and multibody dynamics.

In 2024, Dr. Berger received a Heisenberg grant from the German Research Foundation (DFG). For his exceptional scientific achievements in the field of Applied Mathematics and Mechanics, he received the ``Richard-von-Mises Prize 2021'' of the International Association of Applied Mathematics and Mechanics (GAMM). Dr. Berger further received several awards for his dissertation, including the ``2015 European Ph.D. Award on Control for Complex and Heterogeneous Systems'' from the European Embedded Control Institute and the ``Dr.-Körper-Preis 2015'' from the GAMM. He serves as an Associate Editor for Mathematics of Control, Signals, and Systems, the IMA Journal of Mathematical Control and Information and the DAE Panel, and as a Review Editor for Frontiers in Control Engineering.
\end{IEEEbiography}

\begin{IEEEbiography}[{\includegraphics[width=1.05in,clip,keepaspectratio]{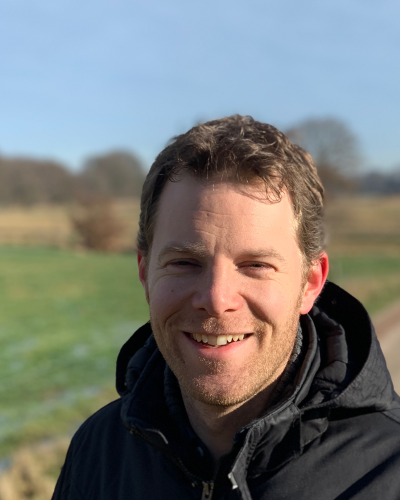}}]{Bart Besselink} received the M.Sc. (cum laude) degree in mechanical engineering in 2008 and the Ph.D. degree in 2012, both from Eindhoven University of Technology, Eindhoven, The Netherlands. Since 2016, he has been with the Bernoulli Institute for Mathematics, Computer Science and Artificial Intelligence, University of Groningen, Groningen, The Netherlands, where he is currently an associate professor. He was a short-term Visiting Researcher with the Tokyo Institute of Technology, Tokyo, Japan, in 2012. Between 2012 and 2016, he was a Postdoctoral Researcher with the ACCESS Linnaeus Centre and Department of Automatic Control, KTH Royal Institute of Technology, Stockholm, Sweden.

His main research interests are on mathematical systems theory for large-scale interconnected systems, with emphasis on contract-based verification and control, model reduction, and applications in intelligent transportation systems and neuromorphic computing. He is a recipient (with Xiaodong Cheng and Jacquelien Scherpen) of the 2020 Automatica Paper Prize.
\end{IEEEbiography}

\end{document}